\documentclass[11pt,a4paper]{article}

\usepackage{xcolor}
\usepackage{a4wide}
\usepackage[utf8]{inputenc}
\usepackage[T1]{fontenc}
\usepackage[fleqn]{amsmath}
\usepackage{amssymb}
\usepackage{amsthm}

\usepackage{natbib}
\renewcommand{\vec}[1]{\bf {#1}}

\newtheorem{theorem}{Theorem}
\newtheorem{lemma}[theorem]{Lemma}
\newtheorem{corollary}[theorem]{Corollary}

\newcommand{\Z}{{\mathbb{Z}}}
\newcommand{\N}{{\mathbb{N}}}
\newcommand{\Q}{{\mathbb{Q}}}
\newcommand{\Jarnik}{Jarn\'{\i}k}
\newcommand{\ellmin}{{\ell_{\text{\it min}}}}

\begin{document}

\title{The Irrationality Exponents of Computable Numbers}
\author{\begin{tabular}{ccc}
    Ver\'{o}nica Becher& Yann Bugeaud &Theodore A. Slaman
    \\
    {\small Universidad de Buenos Aires}&
    {\small Universit\'e de Strasbourg}&
    {\small University of California Berkeley}
    \\
    {\small vbecher@dc.uba.ar}&
    {\small bugeaud@math.unistra.fr}&
    {\small slaman@math.berkeley.edu}
  \end{tabular}
}
\date{August 21, 2014}

\maketitle

Mathematics Subjects Classification: {Primary: 11J04, Secondary: 03Dxx.}


\bigskip \bigskip

The irrationality exponent $a$ of a real number $x$ is the supremum of the set of real numbers~$z$
for which the inequality
\[
0< | x- p/q| < 1/q^z
\]
is satisfied by an infinite number of integer pairs $(p, q)$ with $q > 0$.  Rational numbers have
irrationality exponent equal to $1$, irrational numbers have it greater than or equal to $2$.  The
Thue--Siegel--Roth theorem states that the irrationality exponent of every irrational algebraic
number is equal to $2$.  Almost all real numbers (with respect to the Lebesgue measure) have
irrationality exponent equal to~$2$.  The Liouville numbers are precisely those numbers having
infinite irrationality exponent.

For any real number $a$ greater than or equal to $2$, \citet{Jar31} used the theory of continued
fractions to give an explicit construction, relative to~$a$, of a real number $x_a$ such that the
irrationality exponent of $x_a$ is equal to $a$.  For $a=2$, we can take $x_2 = \sqrt{2}$. For
$a>2$, we construct inductively the sequence of partial quotients of $x_a = [0; a_1, a_2, \ldots
]$. For $n \ge 1$, set $p_n/q_n = [0; a_1, a_2, \ldots , a_n]$.  Take $a_1 = 2$ and $a_{n + 1} =
\lfloor q_n^{a-2} \rfloor$, for $n \ge 1$, where $\lfloor \cdot \rfloor$ denotes the integer part
function. Then, the theory of continued fractions \citep[see][]{Schmidt:1980} directly gives that
the irrationality exponent of $x_a$ is equal to~$a$.

The theory of computability defines a computable function from non-negative integers to non-negative
integers as one which can be effectively calculated by some algorithm.  The definition extends to
functions from one countable set to another, by fixing enumerations of those sets.  A real number $x$ is
computable if there is a base and a computable function that gives the digit at each position of the
expansion of $x$ in that base.  Equivalently, a real number is computable if there is a computable
sequence of rational numbers $(r_j)_{j\geq 0}$ such that $|x -r_j| < 2^{-j}$ for~$j \ge 0$.

The construction cited above shows that for any computable real number~$a$ there is a computable
real number $x_a$ whose irrationality exponent is equal to $a$.  What of the inverse question?  Are
there computable numbers with non-computable irrationality exponents?  Theorem~\ref{1} gives a
characterization of the irrationality exponents of computable real numbers.


\begin{theorem}\label{1}
  A real number $a$ greater than or equal to~$2$ is the irrationality exponent of some computable
  real number if and only if  $a$ is the upper limit of a computable sequence of rational numbers.
\end{theorem}

A real number $x$ is said to be \emph{right-computably enumerable} \citep[see][]{Soare:1969} if and only if
the set of rational numbers $r$ such that $r>x$ is computably enumerable, which is to say that 
$x$ is right-computably enumerable if and only if there is an algorithm to output a listing
$(p_n,q_n)_{n\geq 0}$ of all integer pairs whose quotients are greater than~$x$.  By only
enumerating numbers smaller than any previously enumerated,  one can show that $x$ is
right-computably enumerable if and only if there is a computable strictly decreasing sequence of
rational numbers with limit~$x$.

The set of left-computably enumerable real numbers is defined similarly but with non-decreasing
sequences.  The computable real numbers are exactly those that are both, right and left, computably
enumerable.  There are numbers that are just left-computably enumerable or just right-computably
enumerable.  For example, if $A$ is a computably enumerable but not computable subset of the natural
numbers, such as could be obtained from the Halting Problem, then the real number
$x_A=\sum_{n=1}^{\infty} a_n 2^{-n}$, where for each $n\geq 1$, $a_n=1$ if $n\in A$ and $a_n =0$
otherwise, is left-computably enumerable but not computable.

The theory of computability also considers algorithms that can use external data sets, called
oracles, in the course of their computations.  An oracle is an infinite set of non-negative
integers, and algorithms can ask whether an integer is an element of the set.  The oracle $0'$
encodes all truths of first-order Peano arithmetic that can be expressed with just one block of existential
quantifiers.  For example, the assertion that a polynomial with integer coefficients in several
variables has an integer-valued
solution is a statement of this form.
If a function can be calculated by an algorithm using oracle $0'$ we say that it is computable
in~$0'$.  Similarly, if the set of rational numbers $r$ such that $r>x$ is computably enumerable
in~$0'$ then we say that $x$ is right-computably enumerable in~$0'$, and the equivalence
stated above apply relative to~$0'$.  In Lemma~\ref{4}, we give other equivalences.  In particular,
$x$ is right-computably enumerable in~$0'$ if and only if it is the upper limit of a computable
sequence of rational numbers, which is the condition cited in Theorem~\ref{1}.

Now, consider the case of the irrationality exponent of a computable real number~$x$.  If $x$ is
rational, its irrationality exponent is equal to ~$1$.  If $x$ is irrational algebraic, its
irrationality exponent is equal to~$2$.  In these cases, the irrationality exponents are clearly
right-computably enumerable in~$0'$.  Now, suppose that $x$ is not algebraic.  Then, for every pair
of rational numbers $p/q$ and $b$, $| x-p/q |$ is not equal to $1/q^b$.  Consequently, it is
computable to determine whether $|x-p/q|$ is less than~$1/q^b$ by computing both quantities to
sufficient precision to determine which is larger.  This implies that the set of rational
numbers~$b$ for which there are only finitely many rational numbers~$p/q$ such that $|x-p/q|<1/q^b$
is computably enumerable in~$0'$: Since $x$ is computable, given a rational number $b$ and an
integer $k$, the existential statement ``there are integers $p$ and~$q$ such that $q$ is greater
than~$k$ and $|x-p/q|<1/q^b$\,'' constitutes a single query to~$0'$.  Then,  we can examine all pairs
$b$~and~$k$ (fix one enumeration) and list $b$ upon discovery of some~$k$ for which this query to
$0'$ is answered negatively; that is, there are no $p$ and $q$ such that $q$ is greater than~$k$ and
$|x-p/q|<1/q^b$.  It follows that, if the irrationality exponent of~$x$ is finite, then it is
right-computably enumerable in~$0'$.

Thus, to complete the proof of Theorem~\ref{1}, we only need to show that for every real number~$a$
greater than $2,$
if $a$ is right-computably enumerable in~$0'$, then there is a computable real number~$x$ such that
$x$ has irrationality exponent equal to~$a$.  Since there are numbers that are right-computably
enumerable in~$0'$ that are not computable, the proof of this direction of the theorem, the
existence direction, necessarily involves approximations of sets and real numbers which cannot be
directly computed.  It also immediately implies the following corollary.

\begin{corollary}
  There are computable real numbers whose irrationality exponent is not computable.
\end{corollary}

Similarly and more generally, there are computable real numbers whose irrationality exponents have
no monotonous approximations by a computable sequence of rational numbers.

We give two proofs of Theorem~\ref{1}. The first one is more combinatorial and is based on a
construction given by~\cite{Bugeaud:2008}.  The second one, more geometric and based on a
construction given by~\cite{Jar29}, also yields the following corollary.   

\begin{corollary}\label{3}
  For each real number $a$ greater than or equal to $2$
  and right-computably enumerable in~$0'$, there is a computable Cantor-like
  construction whose limit set contains uncountably many real numbers with irrationality exponent
  equal to $a$, countably many of which are computable.
 \end{corollary}

In fact, the natural measure on this Cantor set concentrates on the set of numbers with the pre-specified
irrationality exponent.

The next lemma states three equivalent, and useful, formulations of the property of right-computable
enumerability in~$0'$.

\begin{lemma}\label{4}
  For any real number  $a$, the following properties are equivalent.
  \begin{enumerate}
  \item There is a  computable  sequence  $(a_j)_{j\geq 0}$ of rational numbers such that $\limsup_{j\geq 0} a_j= a$.

  \item There is strictly decreasing  sequence  $(b_j)_{j\geq 0}$ of rational numbers,  that is
    computable in~$0'$  and  has limit equal to~$a$.

  \item There is a computable doubly-indexed sequence $(a(j,s))_{j,s\geq0}$ of rational numbers satisfying that, 
    for each $j\geq 0$, the sequence $(a(j,s))_{s\geq 0}$ is eventually constant and  the sequence  
    $\big(\lim_{s\to\infty } a(j,s)\big)_{j\geq 0}$ is strictly decreasing   with limit~$a$.
    Without loss of generality, the following can be assumed:

    \begin{enumerate}
    \item The number $a(0,0)$ is an integer greater than or equal to $a(j,s)$,  for   $j\geq 0$ and~$s\geq 0$.

    \item For each $j\geq 0$,  $a(j,0)= a(0,0)$.

    \item For each $s\geq 0$, the sequence $(a(j,s))_{j\geq 0}$ is  strictly decreasing.
    \end{enumerate}
  \end{enumerate}
\end{lemma}

\begin{proof}
  ($1\Rightarrow 2$)~~If $a$ is rational, then the sequence $(b_j)_{j\geq 0}=(a+1/2^j)_{j\geq 0}$
  verifies Condition~2.  Assume that $a$ is not rational and that $(a_j)_{j\geq 0}$ is a computable
  sequence of rational numbers with limit supremum equal to $a$.  Let $M$ be an integer strictly
  greater than each of the values $a_j$, for $j\geq 0$ (this value $M$ may not be found computably
  in~$0'$, but it does exist).  Define $b_0=M$ and $j_0=0$. 

  Let $(c_k)_{k\geq 0}$ be a computable enumeration of the rational numbers.  For $n>0$, let $j_n$
  be the least $j>j_{n-1}$ for which there is a $k<j$ such that $b_{n-1}$ is greater than $c_k$ and
  $c_k$ is greater than the supremum of $(a_j)_{j\geq j_n}$.  Let $b_n=c_k$ for the least such $k$.

  Since $a$ is irrational, $j_n$ is well-defined and, since $j_n$ and $b_n$ are the least integers
  satisfying ``for-all'' properties, they can be computed uniformly in~$0'$.  Thus,
  $(b_n)_{n\geq 0}$ is computable in~$0'$.

  By construction, $(b_n)_{n\geq 0}$ is strictly decreasing and all of its elements are greater than
  $a$.  Let $b$ be the limit of $(b_n)_{n\geq 0}$.  For a contradiction, suppose that $b$ is greater
  than $a$ and consider $c_{k^*}$ for $k^*$ the least index of a rational number strictly between
  $b$ and $a$.  Let $n^*$ be greater than $k^*$ and also so large that $c_{k^*}$ is greater than the
  supremum of $(a_j)_{j\geq j_{n^*}}$.  For every $n$ greater than or equal to $n^*$, $c_{k^*}$
  satisfies the for-all property used to define $b_n$.  But then $(b_n)_{n\geq n^*}$ must be
  contained in $\{c_k:k<k^*\}$, a contradiction.  \medskip

  ($2\Rightarrow 3$)~~Assume $(b_j)_{j\geq 0}$ is computable in~$0'$ with limit $a$.  Let $b_j[s]$
  the computable approximation of the value $b_j$ such that the questions to the oracle $0'$ are
  answered using the set of numbers less than $s$ that are enumerated into $0'$ by computations of
  length less than $s$, if that computation produces a value, and let $b_j[s]$ be $2$, otherwise.
  It follows that, for each $j\geq 0$, there is an integer~$s_j$ 
  such that for every $s\geq s_j$, $b_j[s]= b_j$.

  Let $M$ be an integer greater than $b_0$.  For each $s\geq 0$, we define the sequence
  $(a(j,s))_{j\geq 0}$.  We let $a(0,s)=M$.  
For $j>0$, we let 
\[
a(j,s)=b_j[s]
\] 
provided that for all $k<j$ it holds $b_k[s]>b_{k+1}[s]>2$.  
If this condition fails, then we let 
\[
a(j,s)=(a(j-1,s)+2)/2,
\]
the midpoint between $a(j-1,s)$ and $2$.  
By construction, $a(j,s)$ satisfies conditions
$(a)$, $(b)$ and $(c)$.  
Set ${\tilde s}_{j}=\max\{s_k: k\leq j\}$. 
Then, for each  $k\leq j$, we have $b_k[s]= b_k$
for every  $s\geq {\tilde s}_{j}$.
By hypothesis,  $(b_j)_{j\geq 0}$ is strictly decreasing.
Then,   for each $s\geq {\tilde s}_{j}$,  we deduce that $a(j,s)=b_j$. 
This ensures that the sequence $(a(j,s))_{s\geq
  0}$ is eventually constant and that $(\lim_{s\to\infty}a(j,s))_{j\geq
  0}$ is strictly decreasing with limit $a$.  \medskip

($3\Rightarrow 1$)~~Assume $(a(j,s))_{j,s\geq 0}$ is a sequence of rational numbers such that for
each $j\geq 0$ the sequence $(a(j,s))_{s\geq 0}$ is eventually constant and the sequence
$\big(\lim_{s\to\infty } a(j,s)\big)_{j\geq 0}$ is strictly decreasing with limit $a$.  Let
$\ell(s)$ be the computable function defined by $\ell(0)=0$ and, for $s\geq 1$, let $\ell(s)$ be the
least $j$ less than or equal to $s-1$ such that $a(j,s-1)\neq a(j, s)$, if there is such, and let
$\ell(s)$ be $s-1$ otherwise.  We define the computable sequence $(a_s)_{s\geq 0}$ by
  \[
  a_s=a(\ell(s), s).
  \]
  By assumption on $(a(j,s))_{j,s\geq 0}$, we deduce that
  $\lim_{s\to\infty}\ell(s)=\infty$.  Further, there is an 
  arbitrarily large $t$ with $a_t=a(\ell(t),t)=\lim_{s\to\infty}a(\ell(t),s)$.  Thus, $(a_s)_{s\geq0}$
  and $\big(\lim_{s\to\infty } a(j,s)\big)_{j\geq 0}$ have a common subsequence.  Since
  $\big(\lim_{s\to\infty } a(j,s)\big)_{j\geq 0}$ is strictly decreasing with limit $a$, we get that 
  $\limsup_{s\geq 0}a_s$ is greater than equal to $a$.  Dually, given any number $b$ greater than
  $a$, we can fix $j$ so that $\lim_{s\to\infty } a(j,s)<b$ and fix $t$ so that for all $s>t$,
  $\ell(s)>j$.  Then, for all $s>t$,
  \[
  a_s=a(\ell(s),s)<a(j,s)=\lim_{s\to\infty } a(j,s)<b
  \]
  and so 
\[
\limsup_{s\geq 0}a_s<b,
\] 
as required.
\end{proof}

\section{First proof  of Theorem~\ref{1}}

\begin{proof}[First proof of Theorem~\ref{1}.]

  Let $b \ge 2$ be an integer. Recently, \citet{Bugeaud:2008}
  constructed a class ${\cal C}$
  of real numbers whose irrationality exponent can be read off from their base-$b$ expansion. 
  The class ${\cal C}$ includes the real numbers of the form
  \[
  \xi_{\bf n} = \sum_{j \ge 1} \, b^{-n_j},
  \]
  for a sequence 
  %
  %
  ${\bf n} = (n_j)_{j \ge 1}$ of positive integers satisfying
  $n_{j+1} / n_j \ge 2$ 
  for  every large integer $j$.  
  To obtain good rational approximations to $\xi_{\bf n}$, we
  simply truncate the above sum. Thus, we set
  \[
  \xi_{{\bf n}, J} = \sum_{j = 1}^J  \, b^{-n_j} = {p_J \over b^{n_J}}, \quad J \ge 1.
  \]
  It then follows from
  \[
  \left|\xi_{\bf n} - \frac{p_J}{b^{n_J}}\right| < \frac{2}{(b^{n_J})^{n_{J+1}/n_J}}, 
  \quad J \ge 1,  
  \]
  that the irrationality exponent $\mu (\xi_{\bf n})$ of $\xi_{\bf n}$ satisfies
  \[
  \mu (\xi_{\bf n}) \ge \limsup_{j \to \infty} \, \frac{n_{j+1}}{n_j}\;.
  \]
  \cite{Shallit:1982} 
  proved that the continued fraction expansion of some rational translate  
  of any such $\xi_{\bf n}$ can be given explicitly, and  \cite{Bugeaud:2008} 
  proved that its irrationality exponent is given precisely by
  \[
  \mu (\xi_{\bf n}) = \limsup_{j \to \infty} \, \frac{n_{j+1}}{n_j}\;,
  \]
  and hence can be read off from its  expansion in base $b$. 
  This means that the denominators of the best rational approximations   
  to $\xi_{\bf n} $ are (except finitely many of them) powers of~$b$.   

  Consequently, given a real number $a\geq 2$ for which there is a computable sequence $(a_j)_{j\geq
    0}$ of rational numbers such that $\limsup_{j\to \infty} a_j=a$, it is sufficient to construct a
  computable strictly increasing sequence ${\bf n} = (n_j)_{j \ge 1}$ of positive integers
  satisfying $n_{j+1} / n_j \ge 2$ and
  \[
  \limsup_{j \to \infty} \, \frac{n_{j+1}}{n_j} = a,
  \]
  which we do as follows.  By substituting $2$ for any smaller values, we may assume that each $a_j$
  is greater than or equal to $2$.  We construct the desired sequence ${\bf n}$ by induction as
  follows.  Let $n_1=2$.  Given $n_1,\dots,n_j$, let $n_{j+1}$ be the least $n$ such that $n/n_j\geq
  a_{j+1}$.  By construction, for all $j$, $n_{j+1}/n_j \ge 2$.  Consequently, $n_j\geq 2^j$. 
Since    $(n_{j+1}-1)/n_j<a_{j+1}$,
  $n_{j+1}/n_j -a_{j+1}$ is less than or equal to $1/2^j$.  It follows directly that
  $\limsup_{j\to\infty}n_{j+1}/n_j$ is equal to $\limsup_{j\to \infty} a_j=a$.
\end{proof}

\section{Second proof of Theorem~\ref{1}}

For each real number $a$ greater than $2$, \cite{Jar31} gave a Cantor-like construction of a fractal
subset~$K$ of $[0,1]$ such that the uniform measure $\nu$ on $K$ has the property that the set of
real numbers with irrationality exponent equal to $a$ has $\nu$-measure equal to~$1$.  Thus, for all
real numbers $b$ greater than~$a$, the set of real numbers in $K$ with irrationality exponent equal
to $b$ has $\nu$-measure equal to~$0$.

\begin{lemma}[\protect{\cite{Jar31}}] \label{5} For every real number $a$ greater
  than or equal to $2$, the set of numbers with irrationality exponent equal to~$a$ has Hausdorff
  dimension~$2/a$.
\end{lemma}

We note that \citet{Jar29} and \citet{Bes34} independently established that the set of real numbers
with irrationality exponent greater than or equal to~$a$ has Hausdorff dimension~$2/a$.  Actually,
Lemma~\ref{5} is not explicitly stated in \citet{Jar31}; however, it is an immediate consequence of
the results of that paper.

In the following and throughout this text, we denote by $|I|$ the length of the interval~$I$.

\begin{lemma}[Mass Distribution Principle]\label{6}
  Let $\nu$ be a finite measure, $d$ a positive real number and $X$ a set with Hausdorff dimension
  less than~$d$.  Suppose that there is a positive real number $C$ such that for every interval $I$,
  $\nu(I)<C \ |I|^d$.  Then we have $\nu(X)=0$.
\end{lemma}

\begin{lemma}\label{7}
  Let $\vec{a}$ be a strictly decreasing sequence of rational numbers greater than $2$ which is
  computable in~$0'$ and has limit equal to $a$, greater than~$2$.  
  There is a Cantor-like construction of a fractal $K$, with  uniform measure $\nu$, 
  and a function $C$, computable in~$0'$, from $\Q\cap(0,2/a)$ to $\Q$ such
  that for each rational number $d<2/a$, for every interval $I$,
  $\nu(I) \leq  C(d) |I|^d$.
\end{lemma}

\begin{proof}
  We follow the proof of \Jarnik's Theorem as presented in \cite{Fal03}.
  Let $\vec{a}$ be $(a_j)_{j\geq 0}$.  Fix a computable doubly-indexed sequence $(a(j,s))_{j,s\geq 0}$ of rational numbers
  such that for all $j$, $\lim_{s\to\infty}a(j,s)=a_j$.  Without loss of generality, we assume that
  for every~$s$, the sequence $a(j,s)_{j\geq0}$ is strictly decreasing,  $a(0,0) $ is an
  integer and for all $s$, $a(0,s)=a(0,0)$. Further, we fix a rational number $\beta$ greater than
  $2$ and assume that $\beta$ 
is a lower bound for the numbers $a(j,s)$.  

We fix some notation to be applied in the course of our eventual construction.  For a positive
integer $q$ and a real number $b$ greater than $\beta$, let
  \[
  G_q(b)=\left\{x\in \left(\frac{1}{q^{b}}, 1-\frac{1}{q^{b}}\right): \exists p\in\Z,
    \left|\frac{p}{q} - x\right|< \frac{1}{q^b}\right\}.
  \]
  For $M$ a sufficiently large positive integer according to $\beta$, and $p_1$ and $p_2$ primes
  such that $M< p_1<p_2<2M$, the sets 
  $G_{p_1}(b)$ and $G_{p_2}(b)$ are disjoint and the distance between any
  point in $G_{p_1}(b)$ and any point in $G_{p_2}(b)$ is greater than or equal to
  \[
\frac{1}{4M^2}-\frac{2}{M^b} \geq \frac{1}{8 M^2}.
\]
  For such $M$ the set
  \[
  H_M(b)=\bigcup_{\substack{p \text{ prime}\\M<p< 2M}} G_p(b)
  \]
  is the disjoint union of the intervals composing 
  the sets $G_p(b)$, so $H_{M}(b)$ is made up of intervals of
  length less than or equal to $2/M^{b}$ which are separated by gaps of length at least
  $1/(8M^2)$.  If $I\subseteq[0,1]$ is any interval with $3/|I|< M< p< 2M$ then at least $p|I|/3
  > M |I|/3$ of the intervals in $G_p(b)$ are completely contained in $I$.  By the prime number
  theorem, for sufficiently large $M$ the number of primes between $M$ and $2M$ is at least $M/(2\log
  M)$.  Thus, for such $M$ and $I$, at least $M^2|I|/(6\log M)$ intervals of $H_M(b)$ are contained in
  $I$.  With $M_1$ sufficiently large as above and larger than $3 \times 2^{a(0,0)}$, let
  \[
  M_k= M_{k-1}^k = M_1^{k!}, \ \ \ (k \ge 1).
  \]
  For a positive integer $k$, let $j$ be the least integer less than $k$ such that $a(j+1,k)\neq
  a(j+1,k-1)$, if such exists, and let $j$ be $k-1$, otherwise. That is, $j$ is the greatest index
  less than $k$ such that the approximation to $\vec{a}$ remains unchanged at positions less than or
  equal to $j$ from step $k-1$ to step $k$.  Let
  \[
  b_k= a(j,k).
  \]
  Let $E_0=[0,1]$ and for $k=1,2,\ldots$ let $E_k$ consist of those intervals of $H_{M_k}(b_k)$ that
  are completely contained in $E_{k-1}$.  By discarding intervals if necessary, we arrange that
  all intervals in $E_{k-1}$ are split into the same number of intervals in $E_k$.  The intervals of
  $E_k$ are of length at least $1/(2M_k)^{b_k}$ and are separated by gaps of length at least
  \[
  g_k=\frac{1}{8 M_k^2}.
  \] 
  Thus, each interval of $E_{k-1}$ contains at least $m_k$ intervals of $E_k$ where $m_1=1$ and 
  \[
  m_k=\frac{M_k^2   }{ (2 M_{k-1})^{b_k} 6\log M_k  } \geq \frac{c M_k^2   }{( M_{k-1})^{b_k} \log M_k  }, 
  \]
  if $k\geq 2$ and $c= 1/(2^{a(0,0)} 6)$. 
  Let 
  \[
  K=\bigcap_{k\geq 1}  E_k.
  \]
  Define a mass distribution $\nu$ on $K$ by assigning a mass of 
  $1/(m_1 \times \ldots \times m_{k})$
  to each of the $m_1 \times \ldots \times m_{k}$ many  $k$-level intervals.
  Let $S$ be a subinterval of $[0,1]$. For a lighter notation we write $2\epsilon$ to denote 
  the length $|S|$ of $S$.
  We estimate $\nu(S)$.
  Let $k$ be the integer such that 
  $g_k\leq 2 \epsilon < g_{k-1}$.
  The number of $k$-level intervals that intersect $S$ is 
  \begin{itemize}
  \item at most $m_{k}$, since $S$ intersects at most one $(k-1)$-level interval.
  \item at most $1+ 2\epsilon/g_{k} \leq 4\epsilon/g_{k}$ since the $k$-level intervals have gaps of at least $g_k$ between them.
  \end{itemize}
  Each $k$-level interval has measure $1/(m_1\times\ldots\times m_k)$ so that 
  \[
  \nu(S)\ \leq \ \frac{\min(4\epsilon/g_k,m_k)}{m_1\times\ldots\times m_k}  \ \leq \
  \frac{ (4 \epsilon/g_k)^{ s } \  m_{k}^{1-s}}{m_1\times\ldots\times m_k}, 
  \]
  for every $s$ between $0$ and $1$. 
  Hence,
  \[
  \nu(S)\ \leq\ \frac{2^s}{(m_1\times \ldots\times m_{k-1}) \ m_k^s g_k^s} (2\epsilon)^s.
  \]
  Thus, $\nu(S)$ is at most 
  \begin{align*}
    &
    1
    \frac
    { M_{1}^{b_2} \log M_2  }
    {c M_2^2   }\ \
    \frac{M_2^{b_{3} }\log M_{3}  }
    {c M_{3}^2   }
    \ldots 
    \frac
    { M_{k-2}^{b_{k-1}} \log M_{k-1}  }
    {c M_{k-1}^2   }
    \ \
    \frac{2^s}{m_k^s g_k^s} \ (2\epsilon)^s  =
    \\
    &
    \frac
    { M_{1}^{b_2} \log M_2  }
    {c M_2^2   }\ \
    \frac{ M_2^{b_{3} }\log M_{3}  }
    {c M_{3}^2   }
    \ldots 
    \frac
    { M_{k-2}^{b_{k-1}} \log M_{k-1}  }
    {c M_{k-1}^2   }
    \left(\frac{ M_{k-1}^{b_k} \log M_k  }{c M_k^2   }\right)^s
    ({8 M_k^2})^s
    2^s\
    (2\epsilon)^s  =
    \\
    & \big(\log M_2  \ldots  \log M_{k-1}\big)  
       \big(M_{1}^{b_2} M_{2}^{b_3-2}  \ldots M_{k-2}^{b_{k-1}-2}\big)
    ( \log M_k)^s
    (16 )^s
    c^{-k+2-s}  M_{k-1}^{b_ks-2 }  
    (2\epsilon)^s.  
  \end{align*}
  We want to verify that   
  for every $j$ and for every $s< 2/a_j$
  there is a $C$ such that
  $\nu(S) < C (2\epsilon)^s$. 
  It suffices to show that there is a $C$ such that for every $k$,
  \begin{equation}
\label{{*}}
\tag{*}
  \big(\log M_2  \ldots  \log M_{k-1}\big) 
  \big(M_{1}^{b_2} M_{2}^{b_3-2}  \ldots  M_{k-2}^{b_{k-1}-2}\big)
  ( \log M_k)^s (16 )^s \  c^{-k+2-s} < \ C \   M_{k-1}^{2-b_ks}.
\end{equation}
  Fix $k_0$ such that for every $k\geq k_0,$ $a(j+1,k)=a(j+1,k_0)$.  
  Thus, for every $k\geq k_0,$ $a(j+1,k)=a_{j+1}$.
  Then, define  $\delta>0$ as follows so that 
  for every $k\geq k_0$,   
  \[
  2-\left( b_k \frac{2}{a_j}\right) \geq 2- 2\left( a_{j+1} \frac{2}{a_j}\right)  \geq 2 - 2\frac{a_{j+1}}{a_{j}} = \delta.
  \]
  By the choice of $k_0$ and the definition of $b_k$, for all $k>k_0$, it holds that $b_k<a_{j+1}$.
  Hence the left hand side of the inequality \eqref{{*}}  is at most a constant multiple of
  \[
  \big(\log M_2 \ldots  \log M_{k-1}\big) 
   \big(M_{1}^{a_{j+1}} M_{2}^{a_{j+1}-2}  \ldots  M_{k-2}^{a_{j+1}-2}\big)
   (\log M_k)^s (16 )^s \  c^{-k+2-s}.
  \]
  Furthermore, there is a constant $C$ such that 
  \[
  \big(\log M_2  \ldots  \log M_{k-1}\big) 
  \big(M_{1}^{a_{j+1}} M_{2}^{a_{j+1}-2}  \ldots  M_{k-2}^{a_{j+1}-2}\big)
  (\log M_k)^s (16 )^s \  c^{-k+2-s} 
  < \ C \   M_{k-1}^{\delta}. 
  \]
  The above inequality follows by  noticing that $M_\ell = M_1^{\ell!}$ for $\ell \ge 1$, 
  taking logarithms on each side and recognizing that 
  the contribution of $M_{k-1}$ is the dominating term for sufficiently large $k$.
  The value of $C$ is determined by the value $k_0,$ which is computable in~$0'$ 
  as a function of $j$.
\end{proof}

\begin{proof}[Second proof of Theorem~\ref{1}.]
  Let $a$ be a real number right-computably enumerable in~$0'$ and greater than
  $2$ (for $a$ equal $2$, taking $x$ equals $\sqrt{2}$ suffices).  Fix a computable doubly-indexed sequence
  $(a(j,s))_{j,s\geq 0}$ of rational numbers 
  satisfying property (3) of Lemma~\ref{4}.
 That is, we assume that $\lim_{j\to\infty}
  \lim_{s\to\infty}a(j,s)=a$, for all $s$ the sequence $(a(j,s))_{j\geq 0}$ is strictly decreasing,
  for all $j\geq 0$ the sequence $(a(j,s))_{s\geq 0}$ is eventually constant, 
  for all $s$, we have $a(0,s)=a(0,0)$ and $a(1,s)=a(1,0)$.  
  The last condition gives an appropriate
  initialization of the construction.  Let $K$ be the fractal with measure $\nu$ and $C$ be the
  function associated with this approximation of $a$ in Lemma~\ref{7}.  Fix a computable function
  $C(r,s): \Q\times\N\to \Q$ such that for every $r$ in $\Q$,  $(C(r,s))_{s \ge 0}$ is eventually
  equal to $C(r)$.  We
  may also assume that for all $s$, $C(a(1,s),s)=C(a(1,0),0)$.
  

  We compute a real number $x$ in $K$.  By recursion on $s$ we construct a sequence of nested
  intervals $(I(s))_{s\geq 0}$ such that if $I(s)$ is different from $I(s-1)$ then $I(s)$ is an
  element of the $s$-level of $K$.  We define an auxiliary function $\ell(s)$, with infinite limit,
  to approximate the convergence of the sequence $a(j,s)$.  We also define an auxiliary
  integer-valued function $q(j,s)$ where $j$ is an integer in $[0,\ell(s))$, with the intention that
  $x$ avoids approximation by rational numbers with denominator $q$ greater than or equal to
  $q(j,s)$ within $1/q^{a(j,s)}$.  
  This intention will be realized  in the construction at step $s$ onwards provided that 
  at every  step $t\geq s$, $\ell(t)$ is greater than $j$; in particular, 
  provided  that $a(j,s)$  and $C(a(j,s),s)$ have reached their limit values relative to~$s$.

  We will employ the following estimate.  For a natural number $q_0$ and a real number~$b$ greater
  than or equal to~$2$, let
  \[
  V(q_0,b)= \bigcup_{q\geq q_0} \left\{x\in \left(\frac{1}{q^{b}}, 1-\frac{1}{q^{b}}\right): \exists p\in\Z,
    \left|\frac{p}{q} - x\right|< \frac{1}{q^b}\right\}.
  \]
  Suppose that $b_1>b_2>a$.  By Lemma~\ref{7}, we can estimate $\nu(V(q_0,b_1))$ by
  \begin{align*}
    \nu\big(V(q_0,b_1)\big)&\leq \sum_{q\geq q_0}\sum_{0<p<q} C\big({2}/{b_2}\big)\Big(\frac{2}{q^{b_1}}\Big)^{2/b_2}\\
    &\leq 2 C\big({2}/{b_2}\big)\sum_{q\geq q_0} q \Big(\frac{1}{q^{b_1}}\Big)^{2/b_2}\\
    &\leq 2 C\big({2}/{b_2}\big)\sum_{q\geq q_0} \frac{1}{q^{2 b_1/b_2-1}}.
  \end{align*}
  Thus, for any $\epsilon>0$ there is a $q_0$, uniformly computable from $\epsilon$, $b_1,$ $b_2$
  and $C(2/b_2)$, such that $\nu(V(q_0,b_1))$ is less than $\epsilon$.
  \begin{enumerate}
  \item[] {\em Initial step $0$}. Start with  $I(0)$ equal to the unit interval and $\ell(0)=0.$

  \item[] {\em Step $s$, greater than $0$}. 
    Let $\ell(s)$ be the least $j$  less than or equal to $s$  such  that 
    \[
    a(j+1,s-1)\neq a(j+1, s) \text{ or } C\big({2}/{a(j+1,s)}, s-1\big) \neq C\big({2}/{a(j+1,s)}, s\big)
    \]
    if such exists; otherwise, let $\ell(s)$ be $s$.  By our assumptions on $a(j,s)$ and
    $C(a(1,0),s),$ for every $s>0$, we have that $\ell(s)\geq 1.$

    Let $m(s)$ be the $\nu$-measure given to a level-$s$ interval in~$K$.  We find $h(s)$ so that the
    following inequality holds for each $j$ such that $0\leq j<\ell(s)$,
    \[
    2 C\big({2}/{a(j,s)}, s\big) \sum_{q\geq h(s)} {1}/{q^{\frac{2 a(j,s)}{a(j+1,s)}-1}}<\frac{1}{s}\ \frac{m(s)}{ 2^{s}}.
    \]
    We define $q(j,s)$ for each $j\in [0,\ell(s))$ as follows: if $q(j,s-1)$ is defined then let $q(j,s)=q(j,s-1)$; otherwise, let $q(j,s)=h(s)$.
  
    Let $I(s)$ be  the leftmost level-$s$ interval in $K$  that is included in $I(s-1)$ and satisfies
    \[
    \nu\Big( I(s)\cap \bigcup_{0\leq j< \ell(s)}  
    V\big( q(j, s) , a(j,s)\big)\setminus  V\big(h(s) , a(j,s)\big) \Big)  <  m(s) - 2 \frac{m(s) }{2^{s}} 
    \]
    if such exists; otherwise, let $I(s)$ be $I(s-1)$.
    Note that $m(s) \le \nu (I(s))$. 
  \end{enumerate}

  We now verify that the construction works. Define $\ellmin(s)=\min_{t\geq s}\ell(t)$.  We show by
  induction on $s$ that 
  \[
  \nu\Big(I(s)\cap\bigcup_{0\leq j<\ellmin(s)}V\big(q(j,s),a(j,s)\big)\Big) \leq \nu(I(s))\Big(1-\frac{1}{2^s}\Big).
  \]
  Since $\ellmin(0)=0$, the inductive claim holds for $s=0$.
  Assume the inductive claim for $s-1$:
  \[
  \nu\Big(I(s-1)\cap\bigcup_{0\leq j<\ellmin(s-1)}V\big(q(j,s-1),a(j,s-1)\big)\Big) \leq\nu\big(I(s-1)\big)\Big(1-\frac{1}{2^{s-1}}\Big).
  \]
  Consider those integers $j$ such that $j<\ellmin(s)$.  By the definition of $\ellmin$, we have
  $a(j,s)=\lim_{t\to\infty}a(j,t)$ and $C(2/a(j,s),s)=\lim_{t\to\infty}C(2/a(j,s),t)=C(2/a(j,s)).$
  Further, by the discussion above, 
  \[
  \nu\Big(V\big(h(s),a(j,s)\big)\Big)\leq 2 C\big({2}/{a(j,s)}, s\big) \sum_{q\geq h(s)} 1/q^{\frac{2a(j,s)}{a(j+1,s)}-1}.
  \]
  In the construction we choose $h(s)$ so that for each $j$ less than $\ell(s)$, the term
  on the right side of this inequality is less than $m(s)/(s 2^s)$. 
  This  ensures that for each $j$ less than $\ellmin(s)$,
  the same  upper bound holds for $\nu(V(h(s),a(j,s)))$.

  Now, consider the action of the construction during step $s$.  
  If $I(s)$ is equal to $I(s-1)$, then
  \begin{align*}
   & I(s)\cap\bigcup_{0\leq j<\ellmin(s)}V\big(q(j,s),a(j,s)\big) =\\
    & \qquad \Big(I(s)\,\cap\bigcup_{0\leq j<\ellmin(s-1)}V\big(q(j,s-1),a(j,s-1)\big)\Big)\  \cup \\
    & \qquad \Big(I(s)\,\cap\bigcup_{\ellmin(s-1)\leq j<\ellmin(s)}V\big(q(j,s),a(j,s)\big)\Big).
  \end{align*}
  The first component of the union has $\nu$-measure at most $\nu(I(s))(1-1/2^{s-1})$ and the second
  component has $\nu$-measure at most $m(s)/2^s$.  The union has measure at most
  $\nu(I(s))(1-1/2^{s})$, as required.

  Otherwise, $I(s)$ is a proper subinterval of $I(s-1)$ and satisfies
  \[
  \nu\Big( I(s)\cap \bigcup_{0\leq j< \ell(s)} V\big( q(j, s) , a(j,s)\big)\setminus V\big(h(s),
  a(j,s)\big) \Big) < m(s) - 2 \frac{m(s) }{2^{s}}.
  \]
  Then, 
  \begin{align*}
    &I(s)\cap\bigcup_{0\leq j<\ellmin(s)} V\big(q(j,s),a(j,s)\big)=\\
    &\qquad \Big(I(s)\cap \bigcup_{0\leq j< \ellmin(s)} V\big( q(j, s) , a(j,s)\big)\setminus V\big(h(s) ,
    a(j,s)\big) \Big)\ \  \cup\\
    &\qquad \Big(I(s)\,\cap\bigcup_{0\leq j<\ellmin(s)}V\big(h(s),a(j,s)\big)\Big).
  \end{align*}
  The $\nu$-measure of the first component of the union is less than 
\[
m(s) - 2 \frac{m(s) }{2^{s}}  =\nu\big(I(s)\big)\Big(1 - \frac{1}{2^{s-1}}\Big).
\]  
As in the previous case, the $\nu$-measure of the second component is
  less than $\nu(I(s))/2^s$.  Again, the union has measure at most $\nu(I(s))(1-1/2^{s})$, as required.
  
  It remains to show that there are infinitely many $s$ such that $I(s)$ is a proper subinterval of
  $I(s-1)$.  Consider an $s$ such that $\ell(s)$ is equal to $\ellmin(s)$.  Since 
  \[
  \nu\Big(I(s-1)\cap\bigcup_{0\leq j<\ellmin(s-1)}V\big(q(j,s-1),a(j,s-1)\big)\Big)<\nu\big(I(s-1)\big)\Big(1 - \frac{1}{2^{s-1}}\Big), 
  \]
  we may fix an $s$-level subinterval $I$ of $I(s-1)$ such that 
  \[
  \nu\,\Big(I\,\cap\,\bigcup_{0\leq j<\ellmin(s-1)}V\big(q(j,s-1),a(j,s-1)\big)\Big)< \nu(I)\Big(1 - \frac{1}{2^{s-1}}\Big).
  \]
  For this $I$, 
  \begin{align*}
    &I\cap \bigcup_{0\leq j< \ell(s)}V\big( q(j, s), a(j,s)\big)\setminus  V\big(h(s),
    a(j,s)\big) \subseteq  \\
    &\qquad \Big(I\cap \bigcup_{0\leq j< \ellmin(s-1)} V\big( q(j, s), a(j,s)\big)\Big) \ \cup \\
    &\qquad \Big(I\cap \bigcup_{\ellmin(s-1)\leq j< \ell(s)} V\big( q(j, s), a(j,s)\big)\setminus  V\big(h(s), a(j,s)\big) \Big).
  \end{align*}
  For each $j$ such that $\ell(s-1)\leq j< \ell(s)$, $q(j,s)$ is equal to $h(s)$, so the second
  component of the union is empty.  Thus,
  \begin{align*}
  \nu\Big( I\cap \bigcup_{0\leq j< \ell(s)}  
  V\big( q(j, s) , a(j,s)\big)\setminus  V\big(h(s) , a(j,s)\big) \Big)  &<\  \nu(I)\Big(1 - \frac{1}{2^{s-1}}\Big) 
\\& = \  m(s) - 2 \frac{m(s) }{2^{s}}. 
  \end{align*}
  Hence, the conditions for the construction to define $I(s)$ to be a proper subinterval of $I(s-1)$
  are satisfied, as required.

  Consider the sequence given by the closures of the intervals $I(s), {s\geq0}$.  This is a
  computable nested sequence of intervals whose lengths approach zero in the limit.  Let $x$ be the
  unique real number in their intersection.  By construction, $x$ is computable (as is its base-$b$
  expansion, for every integer $b$ greater than or equal to $2$.)

  We now prove that the irrationality exponent of $x$ is equal to~$a$.  For each $j\geq 0$, let $b_j=\lim_{s\to
    \infty} a(j,s)$. The sequence $(b_j)_{j\geq 0}$ is strictly decreasing with limit~$a$.  The
  construction ensures that for every $j$, there is a step $s$ such that $I(s)$ is a level-$s$
  interval of $K$ containing real numbers that have at least one rational approximation $p/q$ within
  $1/q^{b_j}$.  Thus, the real number $x$ has irrationality exponent greater than or equal to $a$.  We
  now show it can not be greater than $a$.  Suppose that $b$ is greater than $a$.  Let $j$ be such
  that $b$ is greater than $b_j$ and let $s$ be such that $\ellmin(s)$ is greater than $j$.  Then, for
  all $t>s$, $a(j,t)=a(j,s)=b_j$ and $q(j,t)=q(j,s)$.  Further, for any $t>s$,
  $\nu(I(t)\setminus V(q(j,t),b_j))$ is positive.  If there were an integer $q>q(j,s)$ and an integer $p$ such
  that 
\[
\left|x-\frac{p}{q}\right|<\frac{1}{q^b},
\] 
then there would be a $t$ greater than $s$ such that 
\[
I(t)\subset
\Big  (\frac{p}{q}-\frac{1}{q^{b_j}},\frac{p}{q}+\frac{1}{q^{b_j}}\Big).
\]  
But then $I(t)\setminus V(q(j,t),b_j)$ would be empty, a
  contradiction with the fact that it has positive measure.
\end{proof}
\bigskip
\bigskip

\noindent
{\bf Acknowledgements.}
The  authors worked on this problem while they were visiting the Institute for Mathematical Sciences, National University of Singapore, in 2014.  V. Becher is a member of 
Laboratoire International Associ\'e INFINIS, Universit\'e Paris Diderot CNRS - Universidad de Buenos
Aires-CONICET.  Slaman's research was partially supported by the National Science Foundation under grant number DMS-1301659.

\bibliography{ie}
\end{document}